\documentclass[a4paper,11pt]{article}
\usepackage{hyperref}
\usepackage[latin1]{inputenc}
\usepackage{amsmath,amstext,amsthm,amsfonts,amssymb,wasysym,enumerate}
\usepackage{graphicx}

\newtheorem{definition}{Definition}
\newtheorem{proposition}{Proposition}
\newtheorem{theorem}{Theorem}
\newtheorem{corollary}{Corollary}
\newtheorem{lemma}{Lemma}

\title{Deformation of the $\sigma_2$-curvature}
\author{Almir Silva Santos\footnote{almir@mat.ufs.br} \;and Maria Andrade\footnote{maria@mat.ufs.br} \\
\\	
	\vspace{6pt} Universidade Federal de Sergipe, Departamento de Matem\'atica \\S\~ao Crist\'ov\~ao-SE, 49100-000, Brasil }
\date{}

\begin{document}

\maketitle

\begin{abstract}
Our main goal in this work is to deal with results concern to the \linebreak $\sigma_2$-curvature. First we find a symmetric 2-tensor canonically associated to the $\sigma_2$-curvature and we present an Almost Schur Type Lemma.  Using this tensor we introduce the notion of $\sigma_2$-singular space and under a certain hypothesis we prove a rigidity result. Also we deal with the relations between flat metrics and $\sigma_2$-curvature. With a suitable condition on the $\sigma_2$-curvature we show that a metric has to be flat if it is close to a flat metric. We conclude this paper by proving that the 3-dimensional torus does not admit a metric with constant scalar curvature and non-negative $\sigma_2$-curvature unless it is flat.
\end{abstract}

\section{Introduction}
\label{intro}
In 1975, Fischer and Marsden \cite{MR0380907} studied deformations of the scalar \linebreak curvature in a  smooth manifold $M$. They proved several results concern the scalar curvature map $\mathcal R:\mathcal M\rightarrow C^\infty(M)$ which associates to each metric $g\in\mathcal M$ its scalar curvature $R_g$. Here $\mathcal M$ is the Riemannian metric space and $C^\infty(M)$ is the  smooth functions space in $M$. Using the linearization  of the map $\mathcal R$ at a given metric $g$ and its $L^2$-formal adjoint, they were able to prove local surjectivity of $\mathcal R$ under certain hypothesis. Another interesting result concerns flat metrics. They proved that if $g$ has nonnegative scalar curvature and it is close to a flat metric, then $g$ is also flat. Recently, Lin and Yuan \cite{MR3529121} extended many of the results contained in  \cite{MR0380907} to the $Q$-curvature context.

Our goal in this work is to study deformations of the $\sigma_2$-curvature and prove analogous results to those in \cite{MR0380907} and \cite{MR3529121} in this setting. 

Let $(M,g)$ be an $n$-dimensional Riemannian manifold. The $\sigma_2$-curvature, which we will denote by $\sigma_2(g)$, is defined as the second elementary symmetric function of the eigenvalue of the tensor 
$A_g=Ric_g-\frac{R_g}{2(n-1)}g,$ where $Ric_g$ and $R_g$ are the Ricci and scalar curvature of the metric $g$, respectively. A simple calculation gives
\begin{equation}\label{eq002}
\sigma_2(g)=-\frac{1}{2}|Ric_g|^2+\frac{n}{8(n-1)}R_g^2.
\end{equation}

See \cite{arXiv150700053} and the reference contained therein, for motivation to study the  $\sigma_2$-curvature and its generalizations, the $\sigma_k$-curvature. We notice that the definition (\ref{eq002}) and the definition in \cite{arXiv150700053} are the same, up to a constant which depends only on the dimension of the manifold.

In \cite{MR3529121} the authors observed  that if $\gamma_g$ is the linearization of the map $\mathcal R$ and $\gamma_g^*$ is its $L^2$-formal adjoint then $\gamma_g^*(1)=-Ric_g$ and then $R_g=-tr_g\gamma^*_g(1).$

In fact, it is well known that 
$$\gamma_gh=-\Delta_gtr_gh+\delta^2_gh-\langle Ric_g,h\rangle$$
and
$$\gamma_g^*(f)=\nabla^2_gf-(\Delta_gf)g-fRic_g,$$
where $\delta_g=-div_g$ and $h\in S_2(M)$. Here $S_2(M)$ is the space of symmetric $2$-tensor on $M$. 


It is natural to ask whether for each kind of curvature there is a 2-tensor canonically associated to it and if so which information about the manifold we can find through this tensor. In \cite{MR3529121} and \cite{arXiv:1602.01212v2}  the authors addressed this problem in the $Q$-curvature context and have found some interesting results. 

Motivated by the works \cite{MR0380907}, \cite{MR3529121} and \cite{arXiv:1602.01212v2} we consider the $\sigma_2$-curvature as a nonlinear map $\sigma_2:\mathcal M\rightarrow C^\infty(M)$. Let $\Lambda_g:S_2(M)\rightarrow C^\infty(M)$ be the linearization of the $\sigma_2$-curvature at the metric $g$ and $\Lambda_g^*:C^\infty(M)\rightarrow S_2(M)$ be its $L^2$-formal adjoint. After some calculations we obtain the explicit expression of $\Lambda_g$ and $\Lambda_g^*$, see Propositions \ref{prop001} and \ref{lemma001}. We will show that the relation between $\Lambda_g^*(1)$ and the $\sigma_2$-curvature is similar to the relation between the Ricci and scalar curvature. As a first application we derive a Schur Type Lemma for the $\sigma_2$-curvature. Moreover, we derive an almost-Schur lemma similar to that for the scalar curvature and the $Q$-curvature, see \cite{MR2875643} and \cite{arXiv:1602.01212v2}.

We notice here that the $Q$-curvature is defined as 
\begin{equation}\label{eq030}
Q_g=-\frac{1}{n-2}\Delta_g\sigma_1(g)+\frac{4}{(n-2)^2}\sigma_2(g)+\frac{n-4}{2(n-2)^2}\sigma_1(g)^2,
\end{equation}
where $\sigma_1(g)=tr_gA_g$. Although the $\sigma_2$-curvature appears in the $Q$-curvature definition, it is not immediate that the $Q$-curvature and $\sigma_2$-curvature share properties. In fact this is not true. In conformal geometry, for instance, while the $Q$-curvature satisfies a similar equation to the Yamabe equation, the conformal change of the $\sigma_2$-curvature is more complicated, see \cite{MR3420504}, \cite{MR888880} and \cite{arXiv150700053}.

In \cite{CGY}, Chang, Gursky and Yang have defined a $Q$-singular space as a complete Riemannian manifold which has the $L^2$-formal adjoint of the linearization of the $Q$-curvature map with nontrivial kernel. This motivated us to define the following notion of $\sigma_2$-singular space.

\begin{definition}\label{def001}
	A complete Riemannian manifold $(M,g)$ is $\sigma_2$-singular if 
	$$ ker \Lambda_g^*\not=\{0\},$$
	where $\Lambda_g^*: C^{\infty}(M)\to S_2(M)$ is the $L^2$-formal adjoint of $\Lambda_g$. We will call the triple $(M,g,f)$ as a $\sigma_2$-singular space  if $f$ is a nontrivial function in $ker \Lambda_g^*.$
\end{definition}

In Section \ref{sec04} we will prove our first result using this notion.
\begin{theorem}\label{teo001}
	Let $(M^n,g,f)$ be a $\sigma_2$-singular space.
	\begin{enumerate}[(a)]
		\item If $Ric_g-\frac{1}{2}R_gg$ has a sign in the tensorial sense, then the $\sigma_2$-curvature is constant.
		\item If $f$ is a nonzero constant function, then the $\sigma_2$-curvature is identically zero.
	\end{enumerate}
\end{theorem}

In the same direction of the vacuum static spaces and the $Q$-singular spaces we expected to prove that the $\sigma_2$-singular spaces give rise to a small set in the space of all Riemannian manifold. For this purpose we prove that there is no $\sigma_2$-singular Einstein manifold with negative scalar curvature.

\begin{theorem} 
	Let $(M^n, g)$ be a closed Einstein manifold with negative scalar curvature. Then
	$$ker\Lambda^*_g=\{0\}$$
	that is, $(M^n, g)$  cannot be a $\sigma_2$-singular space.
\end{theorem}

On the other hand, if a $\sigma_2$-singular Einsten manifold has positive\linebreak $\sigma_2$-curvature, then  we obtain the following rigidity result.

\begin{theorem} \label{isometricsphere} 
	Let $(M^n,g,f)$ be a closed $\sigma_2$-singular Einstein manifold with \linebreak positive $\sigma_2$-curvature. Then $(M^n,g)$ is isometric to the round sphere with radius $r=\left(\frac{n(n-1)}{R_g}\right)^{\frac{1}{2}}$ and $f$ is an eigenfunction of the Laplacian associated to the first eigenvalue $\frac{R_g}{n-1}$ on $\mathbb S^n(r)$.
\end{theorem}

If the manifold is not $\sigma_2$-singular we can find a locally prescribed  \linebreak $\sigma_2$-curvature problem.

\begin{theorem}\label{teo006}
	Let $(M^n,g_0)$ be a closed Riemannian manifold not $\sigma_2$-singular and satisfying  one of the conditions
	\begin{enumerate}[(i)]
		\item\label{cond001} It is Einstein with positive $\sigma_2$-curvature; or
		\item\label{cond002} 
		$\sigma_2(g)>\displaystyle \frac{1}{2(n-1)}|Ric_g|^2.$
	\end{enumerate}
	Then, there is a neighborhood $U\subset C^{\infty}(M)$ of $\sigma_{2}({g_0})$ such that for any $\psi \in U,$ there is a metric $g$ on $M$ close to $g_0$ with $\sigma_2(g)=\psi.$
\end{theorem}

While the condition (\ref{cond001}) in the theorem can be satisfied by manifolds of any dimension, we notice here that by the Cauchy-Schwarz inequality we have $R_g^2\leq n|Ric_g|^2$. This implies that if a manifold satisfies the condition (\ref{cond002}), which is equivalent to $R_g^2>4|Ric_g|^2$, then its dimension has to be greater than 4. It should be an interesting problem to study what happens in dimension 3 and 4. Also we do not know if the condition (\ref{cond002}) is needed or not, that is, if we can prove the result only with $\sigma_2(g)\geq 0$.



Finally in Section \ref{sec05} we study the relation between flat metrics and the $\sigma_2$-curvature which the first result reads as follows.

\begin{theorem} \label{teo003}
	Let $(M^n,g_0)$ be a closed flat Riemannian manifold with \linebreak dimension $n\geq 3$. Let $g$ be a metric on $M$ with
	$$\int_M\sigma_2(g)dv_g>\frac{1}{8n(n-1)}\int_MR_g^2dv_g.$$ If $||g-g_0||_{C^2(M,g_0)}$ is sufficiently small, then $g$ is also flat.
\end{theorem}

Our last result is concerned with metrics in the 3-torus with constant scalar curvature and nonnegative $\sigma_2$-curvature. This result can be viewed as an extension to $\sigma_2$-curvature context of the result by Schoen-Yau \cite{MR541332}.
\begin{theorem} \label{teo008}
	The $3$-dimensional torus does not admit a metric with constant scalar curvature and nonnegative $\sigma_2$-curvature unless it is flat.
\end{theorem}

The organization of this paper is as follows. In Section \ref{sec01} we do the\linebreak calculations to find the linearization of the $\sigma_2$-curvature map and its \linebreak$L^2$-formal adjoint. Then we find a symmetric 2-tensor canonically\linebreak associated to the $\sigma_2$-curvature and we find an almost-Schur lemma in this context. In Section \ref{sec04} we introduce the notion of $\sigma_2$-singular space and we prove some results. For instance, we prove that under certain hypothesis a $\sigma_2$-singular space has constant $\sigma_2$-curvature. We prove that there is no $\sigma_2$-singular Einstein manifold with negative scalar curvature. Another interesting result in this section is that the only $\sigma_2$-singular Einstein manifold with positive $\sigma_2$-curvature is the round sphere. Moreover we study a local prescribed $\sigma_2$-curvature problem for non $\sigma_2$-singular manifolds. Finally in Section \ref{sec05} we study the relations between flat metrics and $\sigma_2$-curvature. Also we prove a non existence result in the 3-torus.

\section{A symmetric 2-tensor canonically associated to the $\sigma_2$-curvature}\label{sec01}

In this section, we will find a symmetric 2-tensor canonically associated to the $\sigma_2$-curvature which plays an analogous role  to that of the Ricci curvature. Then using the result in \cite{MR3019161} we find directly an almost-Schur type lemma analogous to that in \cite{MR2875643}. Some of the calculations in this section can be found in \cite{MR3412398} and \cite{MR3318510} (see also \cite{K}). For the sake of the reader we do some of the calculations.

Let $(M,g)$ be an $n$-dimensional closed Riemannian manifold with $n\geq 3$. Consider a one-parameter family of metrics $\{g(t)\}$ on $\mathcal M$ with $t\in(-\varepsilon,\varepsilon)$ for some $\varepsilon>0$ and $g(0)=g$. Define $h:=\left.\frac{\partial}{\partial t}\right|_{t=0}g(t)$. It is well-known that the evolution equations of the Ricci and the scalar curvature are
\begin{equation}\label{eq00-1}
\frac{\partial}{\partial t}Ric_g = -\frac{1}{2}\left(\Delta_L h+\nabla^2tr_g  h+2\delta^*\delta h\right)
\end{equation}
and
\begin{equation}\label{eq000}
\frac{\partial}{\partial t}R_g = -\Delta_g tr_gh+\delta^2h-\langle Ric,h\rangle,
\end{equation}
respectively, where 
\begin{equation}\label{eq028}
\Delta_L h:=\Delta h+2\mathring{R}(h)-
Ric\circ h-h\circ Ric
\end{equation}
is the Lichnerowicz Laplacian acting on symmetric 2-tensor. See \cite{MR2274812}, for \linebreak instance. Here $\delta h=-div_g h$, $\delta^*$ is the $L^2$-formal adjoint of $\delta$, $\mathring{R}(h)_{ij}=g^{kl}g^{st}R_{kijs}h_{lt}$ and $(A\circ B)_{ij}=g^{kl}A_{ki}B_{lj}$ for any $A,B\in S_2(M)$.

First we have the following lemma which the proof is a direct calculation.
\begin{lemma}\label{lema001}
	$$
	\displaystyle\frac{\partial}{\partial t}\left|Ric_{g(t)}\right|^2 = 2\left\langle Ric_g,\frac{\partial}{\partial t}Ric_g\right\rangle-2\langle Ric_g\circ Ric_g,h\rangle
	$$
	and
	$$
	\begin{array}{rcl}
	\displaystyle\frac{\partial^2}{\partial t^2}\left|Ric_{g(t)}\right|^2 & = &\displaystyle 4\langle h\circ h,Ric\circ Ric\rangle+2|Ric\circ h|^2
	-8\left\langle Ric\circ h,\frac{\partial}{\partial t}Ric\right\rangle\\
	& & \displaystyle+2\left|\partial_tRic\right|^2+2\left\langle Ric,\frac{\partial^2}{\partial t^2}Ric\right\rangle.
	\end{array}
	$$
\end{lemma}

Thus we can find the linearization of the $\sigma_2$-curvature.

\begin{proposition}\label{prop001}
	The linearization of the $\sigma_2$-curvature is given by
	\begin{eqnarray*}
		\Lambda_g(h) & = & \frac{1}{2}\left\langle Ric_g,\Delta_gh+\nabla^2tr_gh+2\delta^*\delta h+2\mathring{R}(h)\right\rangle\\
		& & -\frac{n}{4(n-1)}R_g\left(\Delta_g tr_gh-\delta^2h+\langle Ric,h\rangle\right).
	\end{eqnarray*}
\end{proposition}
\begin{proof} Let $\{g(t)\}$ be a family of metrics as before. By   (\ref{eq002}), (\ref{eq00-1}), (\ref{eq000}) and Lemma \ref{lema001} we get 
	$$\begin{array}{rcl}
	\displaystyle\frac{\partial }{\partial t}\sigma_2(g) 
	& = & \displaystyle\langle Ric_g\circ Ric_g,h \rangle+\frac{1}{2}\left\langle Ric_g,\Delta_L h+\nabla^2tr_g  h+2\delta^*\delta h\right\rangle\\
	& & \displaystyle-\frac{n}{4(n-1)}R_g\left(\Delta_g tr_gh-\delta^2h+\langle Ric,h\rangle\right)
	\end{array}$$
	and then, by the definition of the Lichnerowicz Laplacian (\ref{eq028}), we obtain the result.
\end{proof}

Therefore we can find the expression of $\Lambda^*_g$ which the proof is a simple calculation using integration by parts and taking compactly supported symmetric 2-tensor $h\in S_2(M)$.

\begin{proposition}\label{lemma001}
	The $L^2$-formal adjoint of the operator $\Lambda_g:S_2(M)\rightarrow C^\infty(M)$ is the operator $\Lambda_g^*:C^\infty(M) \rightarrow S_2(M)$ given by
	$$\begin{array}{rcl}
	\Lambda_g^*(f) & = & \displaystyle\frac{1}{2}\Delta_g(fRic_g)+\frac{1}{2}\delta^2(fRic_g)g+\delta^*\delta(fRic_g)+f\mathring{R}(Ric_g)\\
	& & -\displaystyle\frac{n}{4(n-1)}\left(\Delta_g(fR_g)g-\nabla^2(fR_g)+fR_gRic_g\right).
	\end{array}$$
\end{proposition}

This implies that
\begin{equation}\label{eq003}
tr_g\Lambda_g^*(f)=\frac{2-n}{4}R_g\Delta_gf+\frac{n-2}{2}\langle \nabla^2f,Ric_g\rangle-2\sigma_2(g)f.
\end{equation}

For any smooth vector field $X\in\mathcal X(M)$, if we denote by $X^*$ its dual form, that is, $X^*(Y)=g(X,Y)$ for all $Y\in\mathcal X(M)$, then for any $f\in C^\infty(M)$ we have
$$\int_M\langle X^*,\delta\Lambda_g^*(f)\rangle dv_g=\int_M\langle \delta^*X^*,\Lambda_g^*(f)\rangle dv_g=\frac{1}{2}\int_Mf\Lambda_g(\mathcal L_Xg)dv_g,$$
since $\delta^*X^*=\frac{1}{2}\mathcal L_Xg$. Let $\varphi_t$ be the one-parameter group of diffeomorphisms generate by $X$ such that $\varphi_0=id_M$. Since the $\sigma_2$-curvature is invariant by diffeomorphisms, we have
$\sigma_2({\varphi_t^*g})=\varphi_t^*(\sigma_2(g)).$ This implies that
$$\Lambda_g(\mathcal L_Xg)=D\sigma_2(g)(\mathcal L_Xg)=\langle d\sigma_2(g),X^*\rangle$$
and so
$$\int_M\langle X^*,\delta\Lambda_g^*(f)\rangle dv_g=\frac{1}{2}\int_M \langle fd\sigma_2(g),X^*\rangle dv_g.$$
Therefore
\begin{equation}\label{eq005}
\delta\Lambda_g^*(f)=\frac{1}{2}fd\sigma_2(g).
\end{equation}

By (\ref{eq003}) and (\ref{eq005}) we obtain
\begin{equation}\label{eq004}
tr_g\Lambda_g^*(1)=-2\sigma_2(g)
\end{equation}
and
\begin{equation}\label{eq026}
div_g \Lambda_g^*(1)=-\frac{1}{2}d\sigma_2(g).
\end{equation}

The relations (\ref{eq004}) and (\ref{eq026}) are similar to the relations between the Ricci tensor and the scalar curvature, namely $R_g=tr_gRic_g$ and $div_gRic_g=\frac{1}{2}dR_g$.

In \cite{MR3019161} and \cite{MR3207587}, Cheng generalized the so-called almost-Schur lemma by De Lellis and Topping  \cite{MR2875643} which is close related to the classical Schur's Lemma for Einstein manifolds. As a consequence of (\ref{eq004}), (\ref{eq026}) and the Theorem 1.8 in \cite{MR3207587} we find an almost-Schur type lemma  for the $\sigma_2$-curvature.  

\begin{theorem}
	Let $(M,g)$ be a closed Riemannian manifold of dimension $n\not=4$. Suppose that $Ric_g\geq -(n-1)K$ for some constant $K\geq 0$. Then
	$$
	\int_M(\sigma_2(g)-\overline{\sigma_2(g)})^2\leq\frac{8n(n-1)}{(n-4)^2}\left(1+\frac{nK}{\lambda_1}\right)\int_M\left|\Lambda_g^*(1)+\frac{2}{n}\sigma_2(g)g\right|^2,
	$$
	or equivalently
	$$
	\begin{array}{c}
	\displaystyle\int_M\left|\Lambda_g^*(1)+\frac{2}{n}\overline{\sigma_2(g)}g\right|^2
	\displaystyle\leq \left(1+\frac{16(n-1)}{(n-4)^2}\left(1+\frac{K}{\lambda_1}\right)\right)
	\displaystyle\int_M\left|\Lambda_g^*(1)+\frac{2}{n}\sigma_2(g)g\right|^2
	\end{array}
	$$
	where $\overline{\sigma_2(g)}$ denotes the average of $\sigma_2(g)$ over $M$, $\lambda_1$ denotes the first nonzero eigenvalue of the Laplace operator on $(M,g)$. If $Ric_g>0$  then the equality holds if and only if $\Lambda_g^*(1)=-\frac{2}{n}\sigma_2(g)g$, with constant $\sigma_2$-curvature.
	
\end{theorem}


In \cite{MR3207587} the author has proved an almost-Schur lemma for the $\sigma_k$-curvature in locally conformally flat manifold.  In this case, he used the Newton Transformations associated with $A_g$ and a result by Viaclovsky \cite{V} to get a divergence free tensor and the identity (\ref{eq026}).

\section{$\sigma_2$-singular space}\label{sec04}

In this section we use the definition of the $\sigma_2$-singular space, Definition \ref{def001}, to prove some interesting results. First we obtain that under a certain hypothesis a $\sigma_2$-singular space has constant $\sigma_2$-curvature. Then we classify all closed $\sigma_2$-singular Einstein manifold with $\sigma_2$-curvature. Finally in the spirit of Fischer-Marsden \cite{MR0380907} we prove the local surjectivity of the $\sigma_2$-curvature.


Our first result in this section reads as follows.

\begin{theorem}[Theorem \ref{teo001}]\label{teo004}
	Let $(M^n,g,f)$ be a $\sigma_2$-singular space.
	\begin{enumerate}[(a)]
		\item If $Ric_g-\frac{1}{2}R_gg$ has a sign in the tensorial sense, then the $\sigma_2$-curvature is constant.
		\item If $f$ is a nonzero constant function, then the $\sigma_2$-curvature is identically zero.
	\end{enumerate}
\end{theorem}

\begin{proof}
	In the equation (\ref{eq005}) we obtained that
	$$\mbox{div }\Lambda_g^*(f)=-\frac{1}{2}fd\sigma_2(g).$$
	
	Since $(M^n,g,f)$ is a $\sigma_2$-singular space, then $\Lambda_g^*(f)=0$ and so $fd\sigma_2(g)=0.$ Suppose that there exists an $x_0\in M$ with $f(x_0)=0$ and $d\sigma_2(g)_{x_0}\neq 0.$ By taking derivatives, we can see that  $\nabla^mf(x_0)=0$ for all $\ m\geq 1.$ Moreover, note that by (\ref{eq003}) the function $f$ satisfies 
	\begin{equation}\label{eq013}
	\frac{n-2}{2}\left\langle \nabla^2f,Ric_g-\frac{1}{2}R_gg\right\rangle-2\sigma_2(g)f=0.
	\end{equation}
	
	By hypothesis $Ric_g-\frac{1}{2}R_gg$ has a sign in the tensorial sense, which implies that (\ref{eq013}) is an elliptic equation. By results in \cite{MR0076155} and \cite{MR0086237} we can conclude that $f$ vanishes identically in $M.$ But this is a contradiction. Therefore, $d\sigma_2(g)$ vanishes in $M$ and thus $\sigma_2(g)$ is costant.
	
	Now, if $f$ is a nonzero constant, then we can suppose that the constant is equal to 1. Therefore, by (\ref{eq004}) we obtain
	$$tr_g\Lambda_g^*(1)=-2\sigma_2(g)=0.$$
\end{proof}

\begin{lemma}\label{lemma002}
	Let $(M^n, g)$ be an Einstein manifold with dimension $n\geq 3$, then
	\begin{equation}\label{eq0002}
	\Lambda_g^*(f)=\frac{(n-2)^2}{4n(n-1)}\left(R_g\nabla^2f-R_g(\Delta_gf)g-\frac{R_g^2}{n}fg\right).
	\end{equation}
	Moreover, if $f\in\mbox{ker }\Lambda_g^*$ then
	\begin{equation}\label{eq006}
	R_g\nabla_g^2f+\frac{1}{n(n-1)}R^2_gfg=0.
	\end{equation}
\end{lemma}

\begin{proof}
	Since $(M,g)$ is Einstein, then $Ric_g=\frac{R_g}{n}g$, $R_g$ is constant and
	\begin{equation}\label{eq001}
	\sigma_2(g)=\frac{(n-2)^2}{8n(n-1)}R_g^2.
	\end{equation}
	
	By Proposition \ref{lemma001} we get
	\begin{eqnarray*}
		\Lambda_g^*(f)&=&\frac{1}{2}\Delta_g(fRic_g)+\frac{1}{2}div(div(fRic_g))g+\delta^*\delta(fRic_g)+f\mathring{R}(Ric_g)\\
		&-&\frac{n}{4(n-1)}\left(\Delta_g(fR_g)g-\nabla^2(fR_g)+fR_gRic_g\right)\\
		&=&\frac{R_g}{n}(\Delta_gf)g-\frac{R_g}{n}\nabla_g^2f+f\frac{R_g^2}{n^2}g\\
		& & -\frac{n}{4(n-1)}\left(R_g(\Delta_gf)g-R_g\nabla_g^2f+f\frac{R_g^2}{n}g\right)\\
		&=&-R_g\Delta_gfg\left(\frac{(n-2)^2}{4n(n-1)}\right)+R_g\nabla^2_gf\left(\frac{(n-2)^2}{4n(n-1)}\right)\\
		&-&R_g^2gf\left(\frac{(n-2)^2}{4n^2(n-1)}\right).\\      
	\end{eqnarray*}
	
	Therefore we get (\ref{eq0002}). Now, if $f\in ker\ \Lambda^*_g$, then by (\ref{eq003}) and (\ref{eq001}) we get $R_g\Delta_gf=-\frac{R_g^2}{n-1}f.$ Thus we obtain (\ref{eq006}).
\end{proof}


The trivial examples of $\sigma_2$-singular space are Ricci flat spaces, because  by Proposition \ref{lemma001} we have $\Lambda_g^*\equiv0.$


Let $\mathbb M$ be the {\it round sphere} $\mathbb S^n\subset\mathbb R^{n+1}$ or the {\it hyperbolic space} 
$$\mathbb{H}^n=\left\{(x',x_{n+1})\in\mathbb{R}^{n+1};|x'|^2-|x_{n+1}|^2=-1,x'\in\mathbb{R}^n,x_{n+1}>0\right\}.$$

It is well known that if we take the function $f$ defined in $\mathbb M$ by $f(x)=x_k$ for some $k\in\{1,\dots,n+1\}$, then 
$$\nabla^2 f+\delta gf=0,$$
where $\delta=+1$ in the sphere and $\delta=-1$ in the hyperbolic space. Since the scalar curvature of $\mathbb M$ is $\delta n(n-1)$, then  by Lemma \ref{lemma002} we have $\Lambda_g^*(f)=0$. We point up that these example are also consequence of Tashiro's work \cite{T}.

Up to isometries, we will prove in Theorem \ref{teo005} that the round sphere is the only closed $\sigma_2$-singular Einstein manifold with positive $\sigma_2$-curvature. 

%

\begin{lemma}\label{prop14}
	If $\sigma_2(g)>0$ and $(M^n,g,f)$ is a closed $\sigma_2$-singular Einstein mani-fold, then $R_g>0.$
\end{lemma}

\begin{proof}
	Since $(M,g)$ is an Einstein manifold, then by (\ref{eq001}) we obtain that $R_g\neq 0.$ By Theorem \ref{teo004} the function $f$ is not constant, and by (\ref{eq006}) in Lemma \ref{lemma002}, it satisfies
	
	$$\nabla_g^2f+\frac{1}{n(n-1)}R_gfg=0.$$ 
	Taking trace, we obtain
	$$\Delta_gf+\frac{R_g}{n-1}f=0.$$
	
	Thus
	$$
	\frac{R_g}{n-1}\int_Mf^2dv_g=-\int_Mf\Delta_gfdv_g=\int_M|\nabla_gf|^2dv_g>0,
	$$
	since $f$ is not a constant function. This implies that $R_g>0.$
\end{proof}

By (\ref{eq001}) we see that does not exist Einstein manifold $(M^n,g)$ with negative $\sigma_2$-curvature.  Using Lemma \ref{prop14} we get by contradiction the

\begin{theorem} Let $(M^n, g)$ be a closed Einstein manifold with negative scalar curvature. Then
	$$ker\Lambda^*_g=\{0\}$$
	that is, $(M^n, g)$  cannot be a $\sigma_2$-singular space.
\end{theorem}

Now we will show a rigidity result for closed $\sigma_2$-singular Einstein manifold with positive $\sigma_2$-curvature.

\begin{theorem}[Theorem \ref{isometricsphere}] \label{teo005} 
	Let $(M^n,g,f)$ be a closed $\sigma_2$-singular Einstein manifold with positive $\sigma_2$-curvature. Then $(M^n,g)$ is isometric to the round sphere with radius $r=\left(\frac{n(n-1)}{R_g}\right)^{\frac{1}{2}}$ and $f$ is an eigenfunction of the Laplacian associated to the first eigenvalue $\frac{R_g}{n-1}$ on $\mathbb{S}^n(r)$. Hence 
	$\mbox{dim}\ ker\ \Lambda^*_g=n+1$ and $\displaystyle\int_Mf=0$.
\end{theorem}

\begin{proof}
	By Lemma \ref{prop14} the scalar curvature $R_g$ is a positive constant and by Theorem \ref{teo004} the function $f$ is not constant. As in the proof of Lemma \ref{prop14} we obtain
	\begin{equation}\label{eq020}
	\Delta_gf+\frac{R_g}{n-1}f=0, 
	\end{equation}
	that is, $f$ is an eigenfunction of the Laplacian associated to the eigenvalue $\frac{R_g}{n-1}$.
	On the other hand, by Lichnerowicz-Obata's Theorem, see \cite{MR1333601}, the first nonzero eigenvalue of the Laplacian satisfies the inequality
	\begin{equation}\label{eq007}
	\lambda_1\geq \frac{R_g}{n-1}.
	\end{equation}
	Moreover, the equality holds in (\ref{eq007}) if and only if $(M^n,g)$ is isometric to the round sphere with radius $r=\left(\frac{n(n-1)}{R_g}\right)^{\frac{1}{2}}.$ 
	
	Therefore, by (\ref{eq020}) the ker $\Lambda^*_g$ can be identified with the eigenspace associated to the first nonzero eigenvalue $\lambda_1 > 0$ of Laplacian in the round sphere with radius $r$. But, it is well known that this space has dimension $n+1$. Hence dim ker $\Lambda^*_g=n+1.$
\end{proof}

Since the complex projective space $\mathbb CP^n$ with the Fubini-Study metric is an Einstein manifold with positive $\sigma_2$-curvature, then we see that the condition on the singularity cannot be dropped. Also, since the hyperbolic space is a $\sigma_2$-singular space then the theorem is false if the manifold is only complete. We also can use the Theorem 2 in \cite{T} and (\ref{eq020}) to prove the Theorem \ref{teo005}.

An immediate consequence is the following

\begin{corollary} \label{cor11}
	Let $(M^n,g)$ be a closed Einstein manifold with positive $\sigma_2$-curva-ture. If  $(M^n,g)$ is not isometric to the round sphere, then $(M^n,g)$ is not $\sigma_2$-singular.
\end{corollary}
\subsection{Local Surjectivity of $\sigma_2$}\label{sec02}
In this section we prove a local surjective result to the $\sigma_2$-curvature. To achieve this goal we need of the Splitting Theorem and Generalized Inverse Function Theorem which can be found in \cite{MR0380907}.  The main result in this section reads as follows.

\begin{theorem}[Theorem \ref{teo006}]\label{stability}
	Let $(M^n,g_0)$ be a closed Riemannian manifold not $\sigma_2$-singular. Suppose that
	\begin{enumerate}[(i)]
		\item $g_0$ is an Einstein metric with positive $\sigma_2$-curvature; or
		\item\label{cond003} 
		\begin{equation}\label{equ001}
		\sigma_2(g_0)>\frac{1}{8(n-1)}R_{g_0}^2.
		\end{equation}
	\end{enumerate}
	Then, there is a neighborhood $U\subset C^{\infty}(M)$ of $\sigma_{2}({g_0})$ such that for any $\psi \in U,$ there is a metric $g$ on $M$ close to $g_0$ with $\sigma_2(g)=\psi.$
\end{theorem}

\begin{proof}
	The principal symbol of $\Lambda^*_{g_0}$ is
	\begin{eqnarray*}
		\sigma_{\xi}(\Lambda^*_{g_0})&=&\frac{1}{2}|\xi|^2 Ric_{g_0}+\frac{1}{2}\langle \xi\otimes\xi,Ric_{g_0}\rangle {g_0}-\frac{1}{2}g_0^{ij}(\xi_l\xi_iR_{kj}+\xi_k\xi_iR_{lj})\\
		&-& \frac{n}{4(n-1)}R_{g_0}|\xi|^2g_0+\frac{n}{4(n-1)}R_{g_0} \xi\otimes\xi.
	\end{eqnarray*}
	Taking trace, we get 
	\begin{eqnarray*}
		tr(\sigma_{\xi}(\Lambda^*_{g_0}))=\frac{n-2}{2}\left(-\frac{|\xi|^2}{2}R_{g_0}+\langle \xi\otimes\xi,Ric_{g_0}\rangle_{g_0}\right).
	\end{eqnarray*}
	Thus, $tr(\sigma_{\xi}(\Lambda^*_{g_0}))=0$ implies that 
	\begin{equation}\label{tracexi}
	\frac{|\xi|^2}{2}R_{g_0}=\langle \xi\otimes\xi,Ric_{g_0}\rangle_{g_0}.
	\end{equation} 
	So, we have two cases.
	
	\noindent{\bf Case 1:} If $(M^n,g)$ is Einstein with $\sigma_2(g_0)\neq 0$, then by (\ref{eq001}) we get that $R_{g_0}\neq 0.$ Using the equation (\ref{tracexi}) and $Ric_{{g_{0}}}=\frac{R_{g_{0}}}{n}g$ we get that $\xi\equiv0.$

	\noindent{\bf Case 2:} If $(M^{n},g)$ is not Einstein, then by (\ref{tracexi}) we get
	$$\left|\frac{|\xi|^2}{2}R_{g_0}\right|\leq|\xi\otimes\xi|_{g_0}|Ric_{g_0}|_{g_0}=|\xi|^2_{g_0}|Ric_{g_0}|_{g_0},$$
	which is equivalent to
	$$|\xi|^2_{g_0}\left(\frac{1}{2}|R_{g_0}|-|Ric_{g_0}|_{g_0}\right)\leq 0.$$
	The inequality (\ref{equ001}) implies that $\xi\equiv0$.
	
	Therefore, in any case $\Lambda^*_{g_0}$ has an injective principal symbol. By the {\it Splitting Theorem}, see Corollary 4.2 in \cite{MR0266084}, we obtain that
	\begin{eqnarray*}
		C^{\infty}(M)= Im \Lambda_{g_0}\oplus\ker\Lambda^*_{g_0}.
	\end{eqnarray*}
	Since we assume that $(M,g_0)$ is not $\sigma_2$-singular, then $\ker\Lambda^*_{g_0}=\{0\}$, which implies that $\Lambda_{g_0}$ is surjective.
	
	Therefore, applying the Generalized Implicit Function Theorem, $\sigma_2$ maps a neighborhood of $g_0$ to a neighborhood of $\sigma_2 ({g_0})$ in $C^{\infty}(M).$
\end{proof}

Note that if $(M^n,g)$ is an Einstein manifold with dimension $n>4$, then (\ref{equ001}) holds. Since the round sphere $\mathbb{S}^n$ is Einstein, then a metric in the unit sphere close to the round metric satisfies the condition (\ref{equ001}) and is not $\sigma_{2}$-singular. Also we notice here that for any metric we have $R_g^2\leq n|Ric_g|^2$, thus if the inequality (\ref{equ001}) is satisfied then $n>4$.

As an immediate consequence of the Corollary \ref{cor11} and the Theorem \ref{stability} we obtain the next corollary.

\begin{corollary}
	Let $(M^n,{g_0})$ be a closed Einstein manifold with positive $\sigma_2$-curvature. Assume that $(M^n,{g_0})$ is not isometric to the round sphere. Then, there is a neighborhood $U\subset C^{\infty}(M)$ of $\sigma_{2}({g_0})$ such that for any $\psi \in U,$ there is a metric $g$ on $M$ closed to $g_0$ with $\sigma_2(g)=\psi.$
\end{corollary}

\section{Flat Metrics and the $\sigma_2$-curvature}\label{sec05}

The main goal of this section is to prove the Theorems \ref{teo003} and \ref{teo008}.

Let $(M,g_0)$ be a closed Riemannian manifold. For each $\varepsilon>0$ define the functional
\begin{equation}\label{eq022}
\mathcal{F}_\varepsilon(g)=\int_M\sigma_2(g)dv_{g_0}-\left(\frac{1}{8n(n-1)}+\varepsilon\right)\int_MR_g^2dv_{g_0},
\end{equation}
which is defined in the space $\mathcal M$ of all Riemannian metric in $M$. Note that the volume element is with respect to the fixed metric $g_0$. Next we find the first and second variation of the functional (\ref{eq022}) under a special condition.

\begin{lemma}\label{lema5}
	Let $(M,g_0)$ be a closed flat Riemannian manifold. Let $h$ be a symmetric 2-tensor with $div(h)=0$. Then the first variation of $\mathcal{F}_\varepsilon$ at $g_0$ is identically zero and the second variation is given by
	$$
	D^2\mathcal{F}_\varepsilon(g_0)(h,h)=-\displaystyle\int_M\left(2\varepsilon(\Delta tr(h))^2+\frac{1}{4}\left|\Delta \mathring h\right|^2\right)dv_{g_0},
	$$
	where $\mathring{h}=h-\frac{tr(h)}{n}g$ is the traceless part of $h$.
	
\end{lemma}

\begin{proof} The first variation of $\mathcal F_\varepsilon$ is identically zero because of its definition  (\ref{eq022}) and the metric is flat.
	
	Now consider $g(t)=g_0+th$ for $t$ small enough. Note that by (\ref{eq00-1}) and (\ref{eq000}) we get
	$$
	\frac{\partial}{\partial t}Ric=-\frac{1}{2}\Delta h_{}-\frac{1}{2}\nabla^2tr(h)
	$$
	and
	\begin{equation}\label{eq014}
	\frac{\partial }{\partial t}R_g=-\Delta tr(h).
	\end{equation}

	Next, by Lemma \ref{lema001} we get
	$$
	\displaystyle\frac{\partial^2}{\partial t^2}\left|Ric_{g(t)}\right|^2 =2\left| \frac{\partial }{\partial t}Ric_{g}\right|_{g_0}^2.
	$$

	Using that the metric is flat, we get that $\Delta div =div\Delta $ and $div\nabla^2=\Delta \nabla$. Thus, by (\ref{eq00-1}), (\ref{eq028}) and the fact that $div(h)=0$ we obtain that
	$$\begin{array}{rl}
	\displaystyle\int_M\displaystyle\frac{\partial^2}{\partial t^2}\left|Ric_{g(t)}\right|^2dv_{g_0} & = \displaystyle\frac{1}{2}\int_M\left|\Delta h+\nabla^2tr(h)\right|^2dv_g\\
	& = \displaystyle\frac{1}{2}\int_M\left(\left|\Delta h\right|^2+2\langle \Delta h,\nabla^2tr(h)\rangle+\left|\nabla^2tr(h)\right|^2\right)dv_g\\
	& = \displaystyle\frac{1}{2}\int_M\left(\left|\Delta h\right|^2-\langle div\nabla^2tr(h),\nabla tr(h)\rangle\right)dv_g\\
	& = \displaystyle\frac{1}{2}\int_M\left(\left|\Delta h\right|^2+(\Delta tr(h))^2\right)dv_g.
	\end{array}$$
	If $\mathring h=h-\frac{tr(h)}{n}g$ is the traceless part of $h$, then
	\begin{equation}\label{eq023}
	|\Delta\mathring h|^2=|\Delta h|^2-\frac{(\Delta tr(h))^2}{n}.
	\end{equation}
	Thus
	\begin{equation}\label{eq017}
	\displaystyle\int_M\displaystyle\frac{\partial^2}{\partial t^2}\left|Ric_{g(t)}\right|^2dv_{g_0}=\frac{1}{2}\int_M\left(\left|\Delta \mathring h\right|^2+\frac{n+1}{n}(\Delta tr(h))^2\right)dv_g.
	\end{equation}
	
	Using (\ref{eq014}) we get
	\begin{equation}\label{eq018}
	\frac{\partial^2}{\partial t^2}R_{g(t)}^2=2\left(\Delta tr(h)\right)^2.
	\end{equation}
	
	Finally, by (\ref{eq017}) and (\ref{eq018}) at $t=0$ we have
	$$\begin{array}{rcl}
	\displaystyle\frac{\partial^2 }{\partial t^2}\mathcal F_\varepsilon(g(t)) & = & \displaystyle\int_M\left(\left(\frac{n+1}{8n}+\varepsilon\right)\frac{\partial^2 }{\partial t^2}R_{g}^2-\frac{1}{2}\frac{\partial^2 }{\partial t^2}|Ric_{g}|^2\right)dv_{g_0}\\
	& = & -\displaystyle\int_M\left(2\varepsilon(\Delta tr(h))^2+\frac{1}{4}\left|\Delta \mathring h\right|^2\right)dv_{g_0}.
	\end{array}$$
\end{proof}

In the next result we need the following theorem (See \cite{MR0380907}, \cite{MR3529121} and references contained therein).

\begin{theorem}\label{teo002}
	Let $(M,g_0)$ be a Riemannian manifold. For $ p > n$, let $g$ be  a Riemannian metric on $M$ such that $\|g-g_0\|_{W^{2,p}(M,g_0)}$ is sufficiently small.
	Then there exists a diffeomorphism $\varphi$ of $M$ such that $h := \varphi^*g -g_0$ satisfies that $div_{g_0}(h) = 0$ and 
	$$\|h\|_{W^{2,p}(M,g_0)}\leq c\|g - g_0\|_{W^{2,p}(M,g_0)},$$
	where $c$ is a positive constant which only depends on $(M, g_0)$.
\end{theorem}

Now we are ready to prove the Theorem \ref{teo003}.



\begin{proof}[Proof of the Theorem \ref{teo003}]
	By Theorem \ref{teo002} there exists a diffeomorphism $\varphi$ of $M$ such that if we define $h:=\varphi^*g-g_0$ then $div_{g_0}(h)=0$ and
	$$\|h\|_{C^2(M,g_0)}\leq c ||g-g_0||_{C^2(M,g_0)}.$$
	where the positive constant $c$ depends only on $(M,g_0).$
	
	Expanding $\mathcal{F}_a(\varphi^*g)=\mathcal F_a(g_0+h)$ at $g_0$ and using Lemma \ref{lema5} we obtain
	\begin{equation}\label{eq025}
	\begin{array}{rcl}
	\mathcal{F}_\varepsilon(\varphi^*g)&=&\displaystyle\mathcal{F}_\varepsilon(g_0)+D\mathcal{F}_\varepsilon(g_0)(h)+\frac{1}{2}D^2\mathcal{F}_\varepsilon(g_0)(h,h)+E_3\nonumber\\
	&=&-\displaystyle\displaystyle\int_M\left(2\varepsilon(\Delta tr(h))^2+\frac{1}{4}\left|\Delta \mathring h\right|^2\right)dv_{g_0}+E_3,
	\end{array}
	\end{equation}
	where $|E_3|\leq\displaystyle C\int_M|h||\nabla^2 h|^2dv_{g_0}$ for some constant $C=C(n,M,g_0)>0.$
	Besides, by hypothesis we have
	$$\mathcal{F}_\varepsilon(\varphi^*g)> 0,$$
	for $\varepsilon>0$ small enough.
	
	Now choose $\varepsilon_n>0$ such that $\varepsilon_n<\min\{2n\varepsilon,1/4\}$. Thus using (\ref{eq023}) and (\ref{eq025}) we get
	$$\begin{array}{rcl}
	\displaystyle\varepsilon_n\int_M|\Delta h|^2dv_{g_0}& \leq &\displaystyle\left(2\varepsilon-\frac{\varepsilon_n}{n}\right)\int_M(\Delta tr(h))^2dv_{g_0}+\left( \frac{1}{4}-\varepsilon_n\right)\int_M|\Delta \mathring h|^2dv_{g_0}\\
	&&\displaystyle+\varepsilon_n\int_M|\Delta h|^2dv_{g_0}\\
	&=&2\displaystyle \varepsilon\int_M(\Delta tr(h))^2dv_{g_0} +\frac{1}{4}\int_M|\Delta \mathring h|^2dv_{g_0}\\
	&=&\displaystyle-\mathcal{F}(\varphi^*g)+E_3\leq |E_3|\\
	&\leq& C_0\displaystyle\int_M|h||\nabla^2 h|^2dv_{g_0}.
	
	\end{array}$$
	Suppose $g$ is a Riemannian metric in $M$ such that $||g-g_0||_{C^2(M,g_0)}<\frac{\varepsilon_n}{2cC_0}$. The Theorem \ref{teo002} implies that for $\varepsilon_n>0$ small enough, there exists a diffeomorphism $\varphi$ of $M$ such that taking $h:=\varphi^*g-g_0$ we have $div_{g_0}(h)=0$ and
	$$||h||_{C^0(M,g_0)}\leq||h||_{C^2(M,g_0)}\leq c||g-g_0||_{C^2(M,g_0)}<\frac{\varepsilon_n}{2C_0}.$$
	Therefore
	$$\begin{array}{rcl}
	\displaystyle\varepsilon_n\int_M|\nabla^2 h|^2dv_{g_0} & = & \displaystyle\varepsilon_n\int_M|\Delta h|^2dv_{g_0}\leq C_0\displaystyle\int_M|h||\nabla^2 h|^2dv_{g_0}\\
	& \leq & \displaystyle\frac{\varepsilon_n}{2}\int_M|\nabla^2 h|^2dv_{g_0},
	\end{array}$$
	which implies that $\nabla^2h=0$ on $M.$
	On the other hand, 
	$$\int_M|\nabla h|^2dv_{g_0}=-\int_Mh\Delta h dv_{g_0}=0$$
	and this implies that $\nabla h=0,$ that is, $h$ is parallel with respect to $g_0$.

	Since $g_0$ is flat, then given $p\in M$ we can find local coordinates at $p$ such that $(g_{0})_{ij}=\delta_{ij}$ and $\partial_k(g_{0})_{ij}=0,$ for all $i ,j, k\in\{1,\cdots,n\}$, in some neighbohood $U_p.$ In these coordinates the Christoffel symbols of $\varphi^*g=g_0+h$ are
	$$\displaystyle \Gamma^k_{ij}(\varphi^*g)= \displaystyle\frac{1}{2}(\varphi^*g)^{kl}\left(\nabla_ih_{jl}+\nabla_jh_{il}-\nabla_lh_{ij}\right)=0.
	$$
	
	Therefore, the Riemann curvature tensor is identically zero.
\end{proof}

As a consequence of the previous result we get the following corollary.


\begin{corollary} Let $U\subset \mathbb{R}^n$ be a bounded open set. Let $\delta$ be the canonical metric on $\mathbb{R}^n$. Let $g$ be a metric on $\mathbb R^n$ such that
	\begin{enumerate}
		\item[(i)] $g=\delta$ in $\mathbb R^n\backslash U$;
		\item[(ii)] $||g-\bar{g}||_{C^2(\mathbb{R}^n,\bar{g})}$ is sufficiently small;
		\item[(iii)] $\displaystyle\int_M\sigma_2(g)dv_g>\displaystyle\frac{1}{8n(n-1)}\int_MR_g^2dv_g$.
	\end{enumerate}
	Then $g$ is a  flat metric.
\end{corollary}

\begin{proof}
	Since $U$ is a bounded open set, then we can find a closed retangle $R\subset\mathbb R^n$ such that $U\subset R$. Thus $g=\delta$ in $\mathbb R^n\backslash R$. Identifying the boundary of $R$ properly we obtain the torus $\mathbb T^n$ with one metric satisfying $(iii)$ and  $C^2$-close to the flat metric. By Theorem \ref{teo003} we have that the metric $g$ is flat. 
\end{proof}

Now we will show that in the 3-dimensional torus $\mathbb T^3$ does not exist a metric with constant scalar curvature and nonnegative $\sigma_2$-curvature, unless it is flat. 


\begin{proof}[Proof of the Theorem \ref{teo008}] Suppose that $g$ is a metric with constant scalar curvature and nonnegative $\sigma_2$-curvature. Then by Theorem 5.2 in \cite{MR541332} we obtain that the scalar curvature has to be non positive and is zero if the metric is flat. 
	
	Suppose, without loss of generality, that the constant is $-1$. Then by (\ref{eq002}) we have 
	\begin{eqnarray*}
		\sigma_2(g)=-\frac{1}{2}|Ric_g|^2_g+\frac{3}{16}\geq 0.
	\end{eqnarray*}
	This implies that $|Ric_g|^2_g\leq 3/8.$
	
	Let $p\in M$ be a fixed point. Choose an orthonormal basis $\{e_1,e_2,e_3\}$ for $T_pM$ such that the Ricci tensor at $p$ is diagonal. Let $\{\lambda_1,\lambda_2,\lambda_3\}$ be the eigenvalues of $Ric_g(p)$. Then we have
	$$\lambda_1+\lambda_2+\lambda_3=-1$$
	and 
	$$\lambda_1^2+\lambda_2^2+\lambda_3^2\leq \frac{3}{8}.$$
	
	Thus, for $i\neq j,$ we have
	\begin{eqnarray*}
		0&\geq& \lambda_1^2+\lambda_2^2+\lambda_3^2- \frac{3}{8}\\
		&= & \lambda_i^2+\lambda_j^2+(1+ \lambda_i+\lambda_j)^2-\frac{3}{8}\\
		&=&(\lambda_i+\lambda_j)^2+2(\lambda_i+\lambda_j) +\lambda_i^2+\lambda_j^2+\frac{5}{8}\\
		&\geq&\frac{3}{2}(\lambda_i+\lambda_j)^2+2(\lambda_i+\lambda_j) +\frac{5}{8},
	\end{eqnarray*}
	where in the last inequality we used the inequality  $2(\lambda_i^2+\lambda_j^2)\geq(\lambda_i+\lambda_j)^2.$
	This implies that
	$$12(\lambda_i+\lambda_j)^2+16(\lambda_i+\lambda_j) +5\leq 0.$$
	Since the roots of the equation $12x^2+16x+5=0$ are $-5/6$ and $-1/2$, then  $-5/6\leq\lambda_i+\lambda_j\leq  -1/2$.
	
	Since the Weyl tensor vanishes in dimension 3, we obtain that the decomposition of the curvature tensor, see \cite{MR2274812}, is given by
	$$R_{ijkl}=R_{il}g_{jk}+R_{jk}g_{il}-R_{ik}g_{jl}-R_{jl}g_{ik}-\frac{1}{12}R_g(g_{il}g_{jk}-g_{ik}g_{jl}).$$
	Therefore, we obtain that the sectional curvature of the plane spanned by $e_i$ and $e_j$ satisfies
	$$K(e_i,e_j)=R_{ijji}=R_{ii}g_{jj}+R_{jj}g_{ii}-\frac{1}{2}R_gg_{ii}g_{jj}=\lambda_i+\lambda_j+\frac{1}{2}\leq0.$$
	
	Thus $g$ has nonpositive sectional curvature. However, the torus does not admit a metric with nonpositive sectional curvature, see Corollary 2 in \cite{MR808219}. 
\end{proof}

Lin and Yuan \cite {MR3529121} have proved an analogous result for the $Q$-curvature. In dimension 3 if the scalar curvature is constant by (\ref{eq002}) and (\ref{eq030}) we have
\begin{equation}\label{eq031}
Q_g=4\sigma_2(g)-\frac{1}{2}\sigma_1(g)^2=\frac{23}{32}R_g^2-2|Ric_g|^2.
\end{equation}
In this case if the $Q$-curvature is nonnegative, then the $\sigma_2$-curvature is nonnegative as well. But, by (\ref{eq031}) we see that the sign of  the $\sigma_2$-curvature does not determine the sign of the $Q$-curvature. In particular, the Theorem \ref{teo008} is an extension of the Proposition 5.13 in \cite{MR3529121}.




\end{document}